\documentclass[a4paper,11pt]{amsart}
\usepackage{indentfirst, amsfonts, amsmath, amsthm, amssymb, amscd}
\usepackage{amsmath,amsfonts,amscd,bezier}
\usepackage{hyperref}
\usepackage{color}
\usepackage[active]{srcltx}
\usepackage{mathrsfs}
\newtheorem{maintheorem}{Theorem}

\newtheorem{theorem}{Theorem}[section]

\newtheorem{question}[theorem]{Question}
\newtheorem{corollary}[theorem]{Corollary}
\newtheorem{proposition}[theorem]{Proposition}
\newtheorem{lemma}[theorem]{Lemma}
\newtheorem{definition}[theorem]{Definition}

\theoremstyle{remark}
\newtheorem{remark}[theorem]{Remark}

\title{Equilibrium states for  partially hyperbolic diffeomorphisms with hyperbolic linear part}

\author{J. Crisostomo}
\address{Departamento de Matem\'atica,
  ICMC-USP S\~{a}o Carlos-SP, Brazil.}
\email{jorgemat@icmc.usp.br}

\author{A. Tahzibi} 
\address{Departamento de Matem\'atica,
  ICMC-USP S\~{a}o Carlos-SP, Brazil.}
\email{tahzibi@icmc.usp.br}

\date{}                                         
\begin{document}
\maketitle

\begin{abstract}
We address the problem of existence and uniqueness (finiteness) of ergodic equilibrium states for a natural class of partially hyperbolic dynamics homotopic to Anosov. We propose to study the disintegration of equilibrium states along central foliation as a tool to develop the theory of equilibrium states for partially hyperbolic dynamics. 
\end{abstract}

\section{Introduction}

In this paper we address the problem of existence,  uniqueness (or finiteness) of equilibrium states for a class of partially hyperbolic diffeomorphisms which are homotopic to Anosov (``Derived from Anosov") with respect to  H\"{o}lder potentials. The novelty with respect to previous works in this topic is the use of disintegration of measures along central foliation of partially hyperbolic dynamics. 
\begin{definition} Let $M$ be a closed manifold. A diffeomorphism $f:M\rightarrow M$ is called partially hyperbolic if the tangent bundle $TM$  admits a $Df$-invariant descomposition  $TM=E^{s}\oplus E^{c}\oplus E^{u}$ such that all unit
 vectors $v^{\sigma}\in E^{\sigma}_{x} (\sigma=s, c, u)$ for all $x\in M$ satisfy:

$$\parallel Df(x)v^{s}\parallel<\parallel Df(x)v^{c}\parallel<\parallel Df(x)v^{u}\parallel$$
and moreover $\parallel Df\mid_{E^{s}}\parallel<1$ and $\parallel Df^{-1}\mid_{E^{u}}\parallel<1$.

We called $f$ absolutely partially hyperbolic, if it is partially hyperbolic and for any $x,y,z\in M$ 
$$\parallel Df(x)v^{s}\parallel<\parallel Df(y)v^{c}\parallel<\parallel Df(z)v^{u}\parallel$$  
where $v^{s}\in E^{s}_{x}$, $v^{c}\in E^{c}_{y} $ and $v^{u}\in E^{u}_{z} $.
\end{definition}

For partially hyperbolic diffeomorphisms, it is well known that there are foliations $\mathcal{F}^{\sigma}$ tangent to the sub-bundles  $E^{\sigma}$ for $\sigma=s, u$. The leaf of $\mathcal{F}^{\sigma}$ containing $x$ will be called  $\mathcal{F}^{\sigma}(x)$, for $\sigma=s, u$. In general, it is not true that there is a foliation tangent to the central sub-bundle $E^{c}$. However, by Brin, Burago, Ivanov \cite{Brin-Burago-Ivanov} all absolutely partially hyperbolic diffeomorphism on $\mathbb{T}^{3}$ admit central foliation tangent to $E^{c}$. Hertz, Hertz and Ures \cite{Hertz-Hertz-Ures} gave examples of partially hyperbolic diffeomorphisms which do not admit central foliation. Indeed in their example the central bundle is integrable but not uniquely integrable and there is no foliation tangent to the center bundle.

From now on, by partially hyperbolic in this paper we always mean absolutely partially hyperbolic.

\begin{definition} Consider a continuous map $f:M\rightarrow M$ on a compact manifold $M$. We say that an $f-$invariant Borel probability measure $\mu$ is an equilibrium state for $f$ with respect to a potential $\phi \in C^0(M, \mathbb{R})$  if it satisfies
$$h_{\mu}(f)+\int\phi d\mu=sup\lbrace h_{\eta}(f)+\int\phi d\eta: \eta\in\mathcal{M}(f)\rbrace$$
where $\mathcal{M}(f)$ denotes the set of $f-$invariant Borel probability measures.
\end{definition}

If $\phi\equiv 0$, then any $\mu\in\mathcal{M}(f)$ such that 
$h_{\mu}(f)=\sup\lbrace h_{\eta}(f):\eta\in\mathcal{M}(f)\rbrace$ is called measure of maximal entropy.

\begin{definition} Let $f:\mathbb{T}^{3}\rightarrow\mathbb{T}^{3}$ be an absolutely partially hyperbolic diffeomorphism. $f$ is called Derived from Anosov (DA) diffeomorphism if is homotopic to linear Anosov automorphism $A:\mathbb{T}^{3}\rightarrow\mathbb{T}^{3}$ with three invariant subbundle $T(\mathbb{T}^3) = E^{u} \oplus E^c \oplus E^s$ and $\lambda_1 > \lambda_2 > 0 > \lambda_3$ (or two negative exponents and one positive) are three Lyapunov exponents of $A.$
\end{definition}

Let us observe that in the above definition, by \cite{Hammerlindl} if $f$ is absolutely partially hyperbolic diffeomorphism of $\mathbb{T}^3$ with non-compact center leaves, then its linearization $A$ necessarily is Anosov authomorphism with three distinct eigenvalues. Observe that any Anosov automorphism as above may be considered as a partially hyperbolic diffeomorphism, where weak unstable (or weak stable) bundle will be identified as center bundle. 

 Ures \cite{Ures} proved that DA diffeomorphisms have a unique measure of maximal entropy. Climenhaga, Fisher e Thompson \cite{Climenhaga-Fisher-Thompson}, showed that robuslty transitive diffeomorphisms introduced by Ma\~{n}\'{e} and Bonatti-Viana have a unique equilibrium states with respect to   H\"{o}lder continuous potentials under some mild conditions. Here we  study conditional measures of equilibrium states along central foliation. This enables us to find simpler technical conditions for the uniqueness ( in some cases finiteness) of equilibrium states.

\subsection{Preliminary and statement of results}
Let $f$ be a partially hyperbolic diffeomorphism homotopic to Anosov defined before as DA.  By a well-known result of Franks \cite{Franks} $f$  is semiconjugate to $A$. More specifically, there exists $H:\mathbb{T}^{3}\rightarrow\mathbb{T}^{3}$ homotopic to the identity such that $H\circ f=A\circ H$.

\begin{theorem}[Ures, \cite{Ures}]\label{ee1} Let $f:\mathbb{T}^{3}\rightarrow\mathbb{T}^{3}$ be a DA diffeomorphism. Then, $f$ has a unique measure of maximal entropy.
\end{theorem}

\begin{remark}\label{ee2} Ures in \cite{Ures} showed that:
\begin{itemize}
\item[1)] For all $x\in\mathbb{T}^{3}$, $H^{-1}(x)$ is a compact connected interval (including the case of just a point) in a center manifold.
\item[2)] $C=\lbrace x\in\mathbb{T}^{3}: \#H^{-1}H(x)>1\rbrace$  is  $f-$invariant  and $H^{-1}H(C)=C$. By the previous item we know that $C$ is union of compact intervals. we call these intervals as {\it collapse intervals}. 
\item[3)] If $\mu$ is an $f-$invariant measure, then $h_{\mu}(f)=h_{\nu}(A)$ where $\nu=H_{\ast}(\mu)$.
\end{itemize}
\end{remark}
The existence of equilibrium states for partially hyperbolic systems in $\mathbb{T}^{3}$, associated to any continuous  potential is guaranteed as a consequence of the work of L. Diaz, T. Fisher M. Pacific J. Vieitez \cite{Diaz}.
 However, we can ask whether or not it is true that any H\"{o}lder potential admits a unique equilibrium state for a derived from Anosov diffeomorphism?
The following theorems give a partial answer to this question.

\begin{maintheorem}\label{theo:main.A} 
\label{ee3} Let $f:\mathbb{T}^{3}\rightarrow\mathbb{T}^{3}$ be a DA diffeomorphism and let $\psi:\mathbb{T}^{3}\rightarrow\mathbb{R}$ be a potential H\"{o}lder continuous. Let $\phi:=\psi\circ H$ and $\mu$ be an ergodic equilibrium states for $f$ with respect to  $\phi$:

\begin{itemize}
\item[1)] If $\mu(C)=0$, then $\mu$ is the unique equilibrium state.
\item[2)] If $\mu(C)=1$, then $\mu$ is virtually hyperbolic (see \ref{virtual}) and there exists necessarily another equilibrium state $\eta$ for $(f,\phi)$.
\end{itemize}

\end{maintheorem}
Observe that clearly the first item of the above theorem implies that in the second case (if it occurs) any other equilibrium state give total mass to the union of collapse intervals $C.$

\begin{question}\label{ee10} Is there any  $\psi:\mathbb{T}^{3}\rightarrow\mathbb{T}^{3}$ H\"{o}lder continuous such that $\phi:= \psi \circ H$  verifies the case 2 of the Theorem $\ref{theo:main.A}$?
\end{question}
Currently we do not know $\phi=\psi\circ H$ with $\psi$ H\"{o}lder continuous that satisfies item 2 of the above theorem, although continuous examples exist. Indeed, let $\nu$ be an ergodic  $A-$invariant measure such that $\nu(H(C))=1$. By a result of Ruelle \cite{Ruelle}, there exist a continuous map $\psi:\mathbb{T}^{3}\rightarrow\mathbb{R}$ such that  $\nu$ is an equilibrium state for $(A,\psi)$. Hence, if $\mu$ is $f-$invarint such that $H_{*}\mu=\nu$, then $\mu(C)=1$ and $\mu$ is an equilibrium state for $(f,\phi=\psi\circ H)$.

The proof of the above theorem enables us to conclude a dichotomy between finiteness of ergodic equilibrium states and hyperbolicity of such measures.  Let $f$ and $\phi$ be as before:

\begin{maintheorem}\label{theo:main.B} Let $f$ and $\phi$ be as in Theorem A. Then either there is an ergodic non-hyperbolic equilibrium state or the number of ergodic equilibrium states is finite.
\end{maintheorem}

Recall that $f$ has a unique measure of maximal entropy. Indeed, we can show that
under small variation hypothesis of the potential, the equilibrium state is unique. 

\begin{maintheorem}\label{theo:main.C} 

 Let $f$ and $\phi$ be as in Theorem A. If the potential satisfies $\sup_{\mathbb{T}^3} \phi - \inf_{\mathbb{T}^3} \phi < \lambda_2$, then there exists a unique equilibrium state for $\phi.$
\end{maintheorem}  

The above theorem can be deduced from Theorem A and the following recent result of Viana-Yang

\begin{theorem} \label{yang-viana}
Let $f$ be a DA diffeomorphism and $\mu$ an invariant measure with $h_{\mu} > \lambda_1$ then the disintegration of $\mu$ along central foliation can not be atomic. 
\end{theorem}

The ``small" variational condition in the Theorem C is quite common in the literature of this topic to achieve uniqueness of equilibrium states and it has been considered by K. Oliveira and M. Viana \cite{Oliveira-Viana}, for non-uniformly expanding maps on compact manifolds, I. Rios and J. Siqueira \cite{Rios-Siqueira}, for partially hyperbolic horsehoes, F. Hofbauer and G. Keller \cite{Hofbauer-Keller} , for piece wise monotomic maps, H. Bruin and M. Todd \cite{Bruin-Todd} for interval maps, and M. Denker,  M. Urb\'{a}nski \cite{ Denker-Urbanski} for rational maps on the Riemann sphere.

V. Climenhaga, T. Fisher and J. Thompson \cite{Climenhaga-Fisher-Thompson} proved uniqueness of equilibrium states for natural class of potentials in the setting of Ma\~{n}\'{e} and Bonatti-Viana class of robustly transitive diffeomorphisms. We observe that in one hand their result is more general, as it treats non partially hyperbolic setting. On the other hand, the class under consideration in their result is the special type of systems which are localized perturbations of uniformly hyperbolic dynamics. In fact their result ``gives a quantitative criterion for existence and uniqueness of equilibrium state involving the topological pressure, the norm and variation of the potential, the tail entropy, and the $C^0$ size of the perturbation from the original  Anosov map for the Ma\~{n}\'{e} and Bonatti type examples."

We also mention that R. Spatzier and D. Visscher \cite{SV} proved uniqueness of equilibrium state for frame flows on closed, oriented,  negatively curved $n-$manifold, $n$ odd  and $(n \neq 7)$ and potentials induced by potentials defined on unit tangent bundles, i.e constant on the fibers of the bundle $FM \rightarrow SM$ where $FM$ and $SM$ are respectively frame bundle and unit tangent bundle.
\section*{acknowledgment}
 A.T was enjoying a research period at Universit\'{e} Paris Sud (support of FAPESP-Brasil-2014/23485-2). He thanks  and  J.C was in a 6 months period of his PhD at Universit\'{e} de Bretagne Occidentale (support of FAPESP-Brasil-2015/15127-­1 and 2013/03735-1). The authors  thank warm hospitality of Laboratoire Topologie et dynamique (Orsay) (specially to Jé\'{e}r\^{o}me Buzzi for invitation and conversation) and Laboratoire de Math\'{e}matiques de Bretagne Atlantique (specially to R. Leplaideur).
Both authors thank J. R. Var\~{a}o for conversations.

\section{Disintegration of measures}  \label{disintegrationsection}

In this paper, in order to prove uniqueness (or finiteness) of equilibrium states,  we propose to study the conditional measures of equilibrium states on the leaves of central foliation. 
In what follows we review some basic properties of disintegration of measures.

Let $(M, \mu, \mathcal B)$ be a probability space, where $M$ is a compact metric space, $\mu$ a probability measure and $\mathcal B$ the borelian $\sigma$-algebra.
Given a partition $\mathcal P$ of $M$ by measurable sets, we associate the probability space $(\tilde{M}:= M / \mathcal{P}, \widetilde \mu, \widetilde{\mathcal B})$ by the following way. Let $\pi:M \rightarrow \tilde{M}$ be the canonical projection, that is, $\pi$ associates a point $x$ of $M$ to the partition element of $\mathcal P$ that contains it. Then we define $\widetilde \mu := \pi_* \mu$ and $ \widetilde{\mathcal B}:= \pi_*\mathcal B$.

\begin{definition} \label{definition:conditionalmeasure}
 Given a partition $\mathcal P$. A family $\{\mu_P\}_{P \in \mathcal P} $ is a \textit{system of conditional measures} for $\mu$ (with respect to $\mathcal P$) if
\begin{itemize}
 \item[i)] given $\phi \in C^0(M)$, then $P \mapsto \int \phi \mu_P$ is measurable;
\item[ii)] $\mu_P(P)=1$ $\widetilde \mu$-a.e.;
\item[iii)]
 $\mu = \displaystyle{\int}_{\tilde{M}} \mu_P d\tilde{\mu},$ 
i.e  if $\phi \in C^0(M)$, then 
$\displaystyle{\int}_M \phi d\mu = \int_{\tilde{M}} \int_P \phi d\mu_P  d\tilde{\mu}$.
\end{itemize}
\end{definition}

When it is clear which partition we are referring to, we say that the family $\{\mu_P\}$ \textit{disintegrates} the measure $\mu$. There exists an equivalent form of writing the disintegration formula above: $$\mu = \int_M \mu_x d\mu$$ by considering the conditional measures $\mu_x, x\in M$ where $\mu_x= \mu_y$ if $y \in \mathcal{P}(x).$ In this paper we use both formulation to simplify the notations whenever it is necessary.

\begin{proposition}[ \cite{Einsiedler-Ward}, \cite{Rohlin}]
 If $\{\mu_P\}$ and $\{\nu_P\}$ are conditional measures that disintegrate $\mu$, then $\mu_P = \nu_P$ $\widetilde \mu$-a.e.
\end{proposition}

\begin{corollary} \label{cor:same.disintegration}
 If $T:M \rightarrow M$ preserves a probability $\mu$ and the partition $\mathcal P$, then  $T_*\mu_P = \mu_{T(P)}, \widetilde \mu$-a.e.
\end{corollary}
\begin{proof}
 It follows from the fact that $\{T_*\mu_P\}_{P \in \mathcal P}$ is also a disintegration of $\mu$ and essential uniqueness of system of disintegration.
\end{proof}

\begin{definition} \label{def:mensuravel}
We say that a partition $\mathcal P$ is measurable (or countably generated) with respect to $\mu$ if there exist a measurable family $\{A_i\}_{i \in \mathbb N}$ and a measurable set $F$ of full measure such that 
if $B \in \mathcal P$, then there exists a sequence $\{B_i\}$, where $B_i \in \{A_i, A_i^c \}$ such that $B \cap F = \bigcap_i B_i \cap F$.
\end{definition}

\begin{theorem}[Rokhlin's disintegration \cite{Rohlin}] \label{teo:rokhlin} 
 Let $\mathcal P$ be a measurable partition of a compact metric space $M$ and $\mu$ a borelian probability. Then there exists a disintegration by conditional measures for $\mu$.
\end{theorem}

Let us state a simple but usefull remark which comes from essential uniqueness of disintegration.

\begin{remark} \label{scaling}
Let $(M, \mathcal{B}, \mu)$ be a probability space, $\mathcal{P}$ a measurable partition of $M$ and $X \subset M$ a measurable subset of positive measure. Then $X$ is called $\mathcal{P}-$saturated if  for any $x \in X$ then $\mathcal{P}(x),$ the atom of partition containing $x,$ is contained in $X.$
Let $\mu|_X$ be the normalized (probability) restriction of $\mu$ on $X.$ For any $P \in \mathcal{P}$ such that $P \subset X,$ the conditional measures of $\mu$ and $\mu|_X$ coincide, that is $\mu_P = (\mu|_X)_{P}.$ 

More generally, if $X \subset M$ is a measurable subset with positive measure then $\mathcal{P}$ induces a measurable partition on $X.$ Namely, $$\mathcal{P}_{X}:= \{P_{X}|  P_X:= P \cap X; P \in \mathcal{P}\}$$ is a measurable partition of $X.$ So, by Rokhlin theorem we consider the conditional measures $ (\mu|_X)_{P_X}\}$ obtaining by disintegration of the probability $\mu| X$ on the atoms of partition $\mathcal{P}_X.$ We will use later in the paper the following fact which can be verified using the essential uniqueness of conditional measures: $ (\mu|_{X})_{P_X} = (\mu_P)|_{P_X}$ and consequently $\mu_{P} \leq (\mu|_{X})_{P_X}$ on $P_X \subseteq P.$  
\end{remark}

\subsection{Atomic disintegration along foliations}
In general the partition by the leaves of a foliation may be non-measurable. It is for instance the case for the stable and unstable foliations of  Anosov diffeomorphisms with respect to measures of non vanishing metric entropy. Therefore, by disintegration of a measure along the leaves of a foliation we mean the disintegration on compact foliated boxes. In principle, the conditional measures depend on the foliated boxes, however, two different foliated boxes induce proportional conditional measures. See \cite{Avila-Viana-Wilkinson} for a discussion on this issue. 

\begin{definition}
We say that a foliation $\mathcal F$ has atomic disintegration with respect to a measure $\mu$ if the conditional measures on any foliated box are sum of Dirac measures. 
\end{definition}

Equivalently we could define atomic disintegration as follows: there exist a full measurable subset $Z$ such that $Z$ intersects all leaves in at most a countable set.

Although the disintegration of a measure along a general foliation is defined in compact foliated boxes, it makes sense to say that a foliation $\mathcal F$ has a quantity $k_0 \in \mathbb N$  atoms per leaf. The meaning of ``per leaf'' should always be understood as a generic leaf, i.e. almost every leaf. That means that there is a set $A$ of $\mu$-full measure which intersects a generic leaf on exactly $k_0$ points. 

In the atomic disintegration case, it  may happen that almost all leaves intersect a full meaure set in a non finite but countable  number of points. 

Let us state a recent result of Yang-Viana \footnote{We thank J. Yang for awaring us on the existence of this result when we were working on this paper on the same time.}
\begin{theorem} \label{yang-viana}
Let $f$ be a DA diffeomorphism and $\mu$ an invariant measure with $h_{\mu} > \lambda_1$ then the disintegration of $\mu$ along central foliation can not be atomic. 
\end{theorem}

Let $f$ be a derived from Anosov (or more generally any partially hyperbolic diffeomorphism) diffeomorphism
\begin{definition} \label{virtual}
An $f$-invariant measure $\mu$ is called virtually hyperbolic if there exists a full measurable invariant subset $Z$ such that $Z$ intersects each center leaf in at most one point.
\end{definition}
The above definition was given in \cite{LS} in the context of algebraic automorphisms and the existence of such measures in partially hyperbolic diffeomorphism also had been noticed by (see for instance\cite{Shub-Wilkinson}, \cite{Ponce-Tahzibi-Varao}). 
If $\mu$ is virtually hyperbolic, then the central foliation is measurable with respect to $\mu$ and conditional measures along center leaves are dirac. Indeed the partition into central leaves is equivalent to the partition into points. 

\section{Proof of Theorem \ref{theo:main.A}}
\begin{proof}[Proof of Theorem \ref{theo:main.A}]

From  remark $\ref{ee2}$ we have that  $H_{*}\mu$ is  an equilibrium state for $(A,\psi)$ and it is the unique one since $\psi$ is H\"{o}lder continuous (see \cite{Bowen}).

1. Let us prove the first part of the theorem. So, suppose that $\mu(C) =0.$ If $\nu$ is another equilibrium state for $(f,\psi\circ H)$, then $H_{\ast}\mu=H_{\ast}\nu$. Let $\varphi:\mathbb{T}^{3}\rightarrow\mathbb{R}$ be any continuous map. Since $H^{-1}H(C)=C$ we have $\nu(C)=0$.
Hence, 

$$\begin{array}{lcl}
 \displaystyle{\int}\varphi d\mu &=&\displaystyle{\int}_{\mathbb{T}^{3}\setminus C}\varphi d\mu=\displaystyle{\int}_{\mathbb{T}^{3}\setminus C}\varphi\circ H^{-1}\circ H d\mu=\displaystyle{\int}_{\mathbb{T}^{3}\setminus C}\varphi\circ H^{-1}dH_{\ast}\mu\\
 &=& \displaystyle{\int}_{\mathbb{T}^{3}\setminus C}\varphi\circ H^{-1}dH_{\ast}\nu=\int_{\mathbb{T}^{3}\setminus C}\varphi d\nu\\
 &=&\displaystyle{\int}\varphi d\nu
\end{array}$$
This implies that $\mu=\nu$.

2. Now we prove the second item of the theorem. Let us begin to prove the virtually hyperbolic property of equilibrium state in the case where $\mu(C)=1.$ 
\begin{lemma}\label{ee4}
If $\mu(C)=1$, then $\mu$ is virtually hyperbolic.
\end{lemma}

\begin{proof}
The similar arguments appear  in \cite{Ponce-Tahzibi-Varao2} and here we repeat for completeness. 
Observe that the partition into central leaves is not necessarily a measurable partition and we are not allowed apriori to apply Rokhlin disintegration result to this partition. However, the preimage by $h$ of the partition into points is a measurable partition (see \cite{Ponce-Tahzibi-Varao}). Under the hypothesis of the lemma we just consider the partition into collapse intervals:
 $$\mathcal{N}:= \{H^{-1}(x): H^{-1}(x) \quad \text{is a non trivial closed interval}  \}$$ So we can speak about disintegration of $\mu$ along collapse intervals. We denote by $\mu_{\mathcal{N}(x)}$ the conditional measure supported on the collapse interval containing $x.$ Of course, if $\mathcal{N}(x) = \mathcal{N}(y)$ then $\mu_{\mathcal{N}(x)} = \mu_{\mathcal{N}(y)}.$

Fix an orientation for the central leaves and for each collapse interval (any element of the partition $\mathcal{N}$) and consider the left extreme point of them. It can be proved that that  the left extreme point of collapse intervals form  a measurable set (see \cite{Ponce-Tahzibi-Varao2}). We call these sets as point zero, that is if $x \in \mathcal C$ then $0_x$ means the left extreme point associated to the segment $\mathcal{N}(x)$ where $\mathcal{N}(x) \in \mathcal{N}$ which contains $x.$ If $y \in \mathcal{N}(x)$ then $[0_x,y]$ stands for the segment inside the center leaf which contains $0_x$ and $y$.

We now consider the set
\[H_\alpha=\{ y: [0_x,y] \subset \mathcal{N}(x) \;|\; \mu_{\mathcal{N}(x)}([0_x,y])\leq \alpha\}\]

We claim that $H_\alpha$ is an invariant set. This comes from the fact that $f(\mathcal{N}(x)) = \mathcal{N}(f(x))$ and $f_{*} \mu_{\mathcal{N}(x)} = \mu_{\mathcal{N}(f(x))}$. Hence $H_\alpha$ is an invariant set. From the definition of disintegration and $H_\alpha$ notice that $\mu (H_\alpha) \leq \alpha$. Since we are assuming that we are not in the atomic case, there must exist $\alpha_0 \in (0,1)$ such that $0 < \mu (H_\alpha) \leq \alpha_0 < 1$. By ergodicity we would have $\alpha \geq \mu (H_{\alpha_0})=1$, absurd.

Thus, the disintegration is indeed atomic. In fact there exist at most one atom in each collapse interval. Indeed, let $A_{n}$ be the set of atoms with weight belonging to the interval $[\frac{1}{n+1}, \frac{1}{n}).$ As disintegration is unique and $\mu$ is invariant, $A_{n}$ is invariant and by ergodicity and usual measure theory argument we get that all of the atoms have full weight. Consequently there is at most one atom in each collapse interval. So, we get a full measurable subset  $\mathcal{M} \subset \mathbb{T}^3$ such that intersects each center leaf in at most a countable number of points. Observe that $H$ restricted to $\mathcal{M}$ is injective.

Now the idea is to use theorem $B$ of \cite{Ponce-Tahzibi-Varao} and conclude that we have exactly one atom per (global) leaf. Although their result is for volume measure, it applies also in our setting. Let us review the main arguments. First of all we show that the number of atoms on  central leaves is bounded. By this we mean that there exist a full measurable subset which intersects all center leaves in a finite (uniformly bounded) number of points. 

By contradiction, suppose that his is not the case. So every full measurable subset of $\mathcal{M}$ intersects any typical center leaves in infinitely many points. Define $\nu := H_* \mu$ and observe that it is an invariant measure by the linear hyperbolic automorphism. Any measurable subset of $H(\mathcal{M})$ of $\nu-$full measure  intersects almost all  center leaves in an infinite (countable) number of points. 

Let $\{R_i\}$ be a Markov partition for $A$ and consider the partition $ \mathcal{P}:= \{\mathcal{F}_{R}^c(x), x \in R_i \quad \text{for some $i$}\}$ where $\mathcal{F}_{R}^c(x)$ denotes the connected component of $\mathcal{F}^c(x) \cap R_i$ and contains $x$ in its interior. The partition $\mathcal{P}$ is a measurable partition and by Rokhlin theorem we can disintegrate $\nu$ along the elemnts of this partition.   As $\nu$ an equilibrium state for Anosov automorphism, it gives zero mass to the boundary of markov partition.  Let $\nu_x$ be the conditional measure supported on $\mathcal{F}_{R}^c(x).$  Observe that, as $H(\mathcal{M})$ intersects typical leaves in a countable number of points, the conditional measures $\nu_x$ should be atomic. 

\begin{proposition}
There is a natural number $\alpha_0 \in \mathbb{N}$ such that for $\nu-$almost every $x,$  $\nu_x$ contains exactly $\alpha_0$ atoms.   
\end{proposition}

\begin{proof}
 Firstly we observe that:
 \begin{lemma}
$A_*\nu_x \leq \nu_{A(x)}$ restricted to the subsets of $\mathcal F^c_{R}(A(x)).$
\end{lemma}

Observe that $A_{*} \nu_x$ and $\nu_{A(x)}$ are probability measures defined respectively on $A(\mathcal{F}_{R}^c(x))$ and $\mathcal F^c_{R}(A(x)).$ Fix an element of Markov partition $R_i.$ By Remark \ref{scaling}, $\nu_x, x \in R_i$ coincides with the disintegration of the normalized restriction of $\nu$ on $R_i$ which we denote by $\nu|_{R_i}.$  As $\nu$ is invariant $A_* (\nu|_{R_i}) = \nu|_{A(R_i)}$, by essential uniqueness of disintegration, $A_* \nu_x$ concides with the disintegration of $\nu|_{A(R_i)}$ along the partition $A(\mathcal{F}_{R}^c(x)), x \in R_i.$   

For any $j$ such that $A(R_i) \cap R_j \neq \emptyset,$ by Markov property $A(R_i)$ crosses $R_j$ completely in the center-unstable direction and so for all $x \in R_i,$
$$ \mathcal{F}^c_{R}(A(x)) \subset A(\mathcal{F}_{R}^c(x)).$$
 Again by remark \ref{scaling} we conclude that $ A_{*} \nu_x \leq \nu_{A(x)}$ on $\mathcal{F}^c_{R} (A(x)).$


 Given any $\delta \geq 0$ consider the set $K_\delta =\{ x \in \mathbb{T}^3 \; | \; \nu_{x}(\{x\}) > \delta \}$, that is, the set of atoms with weight at least $\delta$. 
 If $x\in K_{\delta}$ then
 $$\delta < \nu_x(\{x\}) =  A_{*}\nu_x(\{A(x)\}) \leq \nu_{A(x)} (\{A(x)\}).$$
 
 Thus $A(K_{\delta}) \subset K_{\delta}$, and by the ergodicity of $A$ we have that
$\nu(K_\delta)$ is  zero or one, for each $\delta \geq 0$. Note that $\nu(A_0)=1$ and $\nu(A_1)=0$. Let $\delta_0$ be the critical point for which $\nu(A_\delta)$ changes value, i.e, $\delta_0 = \sup \{\delta : \nu(K_{\delta}) = 1\}$. This means that all the atoms have weight $\delta_0$.  Due to the atomicity of disintegration, the value of $\delta_0$ has to be a strictly positive number.  Since $\nu_x$ is a probability we have an $\alpha_0:=1/\delta_0$ number of atoms as claimed.
\end{proof}

In particular the above lemma shows that given a fixed length $L \in \mathbb{R}$ there exist $N \in \mathbb{N}$ such that the number of atoms in any typical center plaque of size $L$ is at most $N.$ Recall that we had supposed that $H (\mathcal{M})$  intrinsically intersects center leaves in infinitely many points. So, take a center plaque $D \subset \mathcal{F}^c_x$ with more than $N$ atoms. By backward contraction along central leaves by $A$ we get a large $n > 0$ such that the length of $A^{-n} (D)$ is less than $L$. As $\nu$ is invariant and disintegration is unique we get a center plaque with length less than $L$ containing more than $N$ atoms which is absurd. 
The proof of lemma is complete.

\end{proof}
Now we complete the proof of the second item of theorem. Up to now, we have seen that if $\mu (C) =1$ then central foliation is measure theoretically  equivalent to the partition of $\mathbb{T}^3$ into points and consequently measurable. We denote by $(\tilde{M}, \tilde{\mu})$ the quotient space $\mathbb{T}^{3} / \mathcal{F}^c$ equipped with the quotient measure. Observe that by virtual hyperboliciy proved above, any element $\tilde{x} \in \tilde{M}$ can be considered as a unique collapse interval inside the center leaf $\mathcal{F}^c(x).$ From now on we denote this collapse interval by $\mathcal{N}(\tilde{x}).$  We denote by $\tilde{f} : \tilde{M} \rightarrow \tilde{M}$ the induced map on the quotient space. Clearly as $\mu$ is invariant by $f$ then $\tilde{\mu}$ obtained by natural quotient is invariant by $\tilde{f}.$

Now, by lemma \ref{ee4}, we have that $$\mu=\displaystyle{\int}\delta_{a (\tilde{x})} d\tilde{\mu} (\tilde{x})$$ where $a (\tilde{x})\in \mathcal{N}(\tilde{x})$ and $\mathcal{N}(\tilde{x})$ is the collapse interval corresponding to $\tilde{x}$. Let $b (\tilde{x}) \neq a (\tilde{x})$ be the left extreme point of $\mathcal{N}(\tilde{x})$. We define $$\eta=\int\delta_{b (\tilde{x})}d\tilde{\mu} (\tilde{x})$$
$\eta$ is well defined because $\lbrace b (\tilde{x})\rbrace$ is a measurable set. We claim that $\eta$ is $f-$invariant measure and $H_{\ast}\mu=H_{\ast}\eta$, and ergodicity of $\mu$ implies ergodicity of $\eta$. To show that $\eta$ is invariant take any continuous $\xi$ and observe that:

\begin{align*}
\int \xi \circ f d \eta &= \int \int \xi\circ f d \delta_{b(\tilde{x})} d \tilde{\mu}  = \\ \int \xi (f (b ( \tilde{x}))) d \tilde{\mu} &= \int \xi(b(\tilde{f}(\tilde{x}))) d \tilde{\mu} 
= \int \xi(b(\tilde{x})) d \tilde{\mu}  = \int \xi d \eta  .
\end{align*}
 where the third equality comes from the invariance of collapse intervals and that $f$ preserves orientation on center foliation. The fourth equality is consequence of invariance of $\tilde{\mu}$ by $\tilde{f}.$

To prove the ergodicity of $\eta,$ consider any invariant subset $D$ with $\eta(D) >0.$ Observe that $\tilde{\mu}$ is ergodic as an invariant measure of $\tilde{f}.$ As $f(b(\tilde{x})) = b(\tilde{f} (\tilde{x}))$ and $D$ is invariant we have that $\{ \tilde{x}: \chi_D (b(\tilde{x})) = 1\}$ is an $\tilde{f}$ invariant subset of $\tilde{M}.$  So, ergodicity of $\tilde{\mu}$ implies that it has full measure. This implies that $\eta (D) =1.$

  By essential uniqueness of disintegration  $\eta\neq\mu$.


On the other hand, as $H(a (\tilde{x})) = H(b(\tilde{x}))$ and $\phi = \psi \circ H$ we have:

$$\int\phi d\eta =\int\phi(b (\tilde{x}))d\tilde{\mu}=\int\phi(a(\tilde{x}))d\tilde{\mu}=\int\phi d\mu.$$

This implies that $\eta$ is an equilibrium state for $(f,\phi)$.

\end{proof}
	\section{Proof of Theorem \ref{theo:main.B}}

\begin{lemma}\label{ee9} Let $\mu=\displaystyle{\int}\delta_{a(\tilde{x})} d\tilde{\mu}(\tilde{x})$ and $\nu=\displaystyle{\int}\delta_{b (\tilde{x} )}d\tilde{\mu}(\tilde{x})$ be $f-$invariant. If  $\displaystyle\lim_{n\rightarrow\infty} d(f^{n}(a(\tilde{x})), f^{n}(b(\tilde{x}))=0$, then $\mu=\nu$.
\end{lemma}
\begin{proof} Since $f_{*}\mu=\mu$ and $f_{*}\nu=\nu$, we have
 $\mu=f^{n}_{*}\mu=\displaystyle{\int}\delta_{f^{n}(a(\tilde{x}))}d\tilde{\mu}$ and  $\nu=f^{n}_{*}\nu=\displaystyle{\int}\delta_{f^{n}(b(\tilde{x}))}d\tilde{\mu}$. 
 
 Let $\varphi:\mathbb{T}^{3}\rightarrow\mathbb{R}$ be any Lipschitz map. Hence, 

$$\begin{array}{lcl}
 \mid\displaystyle{\int}\varphi d\mu -\displaystyle{\int}\varphi d\nu\mid &=&\mid \displaystyle{\int}\varphi(f^{n}(a(\tilde{x}))d\tilde{\mu}-\displaystyle{\int}\varphi(f^{n}(b(\tilde{x}))d\tilde{\mu}\mid\\
 &\leq&\displaystyle{\int}\mid \varphi(f^{n}(a(\tilde{x}))-\varphi(f^{n}(b(\tilde{x}))\mid d\tilde{\mu}\\
& \leq&\displaystyle{\int} kd( f^{n}(a(\tilde{x})), f^{n}(b(\tilde{x})))d\tilde{\mu}
 \end{array}$$
This implies that $\displaystyle{\int}\varphi d\mu =\displaystyle{\int}\varphi d\nu$, since that Lipschitz map set is dense  in $C(\mathbb{T}^{3})$, we have that last equality is holds for $\varphi\in C(\mathbb{T}^{3})$. 
\end{proof}

\begin{proof}[Proof of Theorem \ref{theo:main.B}]  To prove the dichotomy in the statement of the theorem, suppose that there does not exist any ergodic equilibrium state with zero central Lyapunov expnent. Now, by contradiction suppose that there is a sequence of ergodic equilibrium states  $\lbrace\mu_{n}\rbrace$ for $(f,\phi)$. By theorem \ref{theo:main.A}  we have that  $\mu_{n}(C)=1$. By Lemma $\ref{ee4}$ all $\mu_n$ are virtually hyperbolic. Observe that for all $n,$ $H_* \mu_n = \nu$ where $\nu$ is the unique equilibrium state of for $(A, \psi).$ 

By uniqueness of disintegration we conclude that the dirac conditional  measures of $\mu_n$ are push forwarded to dirac disintegration of $\nu$ along central leaves of $A.$ This implies that  $\tilde{\mu}_m =  \tilde{\mu}_n$ for all $n, m$ where $\tilde{\mu}_n$ is the quotient measure obtained from the disintegration of $\mu_n$ along central foliation (see section \ref{disintegrationsection}).  In other words, all $\mu_n$ are virtually hyperbolic and for any two $m, n$ there exist full measurable subsets  $Z_m, Z_n$, $\mu_m(Z_m)= \mu_n(Z_n) = 1$  such that $Z_m$ and $Z_n$ intersect almost all center leaves in a unique point and the intersection point belong to the same collapse interval. So,  
\begin{center}
$\mu_{n}=\displaystyle{\int} \delta_{a_n (\tilde{x} ) } d\tilde{\mu}.$ 
\end{center} 
where $\tilde{\mu}$ stands for the quotient measure for all $\mu_n$. We emphasize that the dirac masses $a_n(\tilde{x})$ are in the same collapse interval for all $n.$ Now by compactness of collapse intervals, let
$a (\tilde{x})\in \mathcal{N} (\tilde{x})$ such that $\lim_{n\rightarrow\infty}a_n (\tilde{x})=a (\tilde{x})$. Since for any $n,$ $\lbrace a_n(\tilde{x})\rbrace$ is a measurable invariant set it comes out that  $\lbrace a(\tilde{x})\rbrace$ is measurable and invariant. Define, 
\begin{center}
$\eta=\displaystyle{\int}\delta_{a(\tilde{x})}d\tilde{\mu} (\tilde{x})$ 
\end{center}
Thus,
$$\lim_{n\rightarrow\infty}\mu_{n}=\eta$$ 

Since $f$ is entropy expansive (see \cite{Diaz}), we have $$\limsup_{n\rightarrow\infty}h_{\mu_{n}}(f)\leq h_{\eta}(f)$$
Hence,
$$\limsup_{n\rightarrow\infty}h_{\mu_{n}}(f)+\int\phi d\mu_{n}\leq h_{\eta}(f)+\int\phi d\eta$$ 
This implies that $\eta$ is an equilibrium state for $(f,\phi)$. Then, $H_{*}\mu_{n}=H_{*}\eta$ and thus $\eta(C)=1$. 
 
Since,
 $$\lambda^{c}(\mu_{n})=\int \parallel Df\mid_{E^{c}}\parallel d\mu_{n}$$ we have,
$$0\geq\lim_{n\rightarrow\infty}\lambda^{c}(\mu_{n})=\lim_{n\rightarrow\infty}\int \parallel Df\mid_{E^{c}}\parallel d\mu_{n}=\int \parallel Df\mid_{E^{c}}\parallel d\eta=\lambda^{c}(\eta)$$

We claim that $\lambda^{c}(\eta)=0$. If not,  $\lambda^{c}(\eta)<0$ and by $\lim_{n\rightarrow\infty}a_n (\tilde{x})=a (\tilde{x})$ there exist $n_{0}$ such that $a_n (\tilde{x})$ belong to local stable manifold of $a (\tilde{x})$, for $n\geq n_{0}$. By Lemma $\ref{ee9}$, we have that $\eta=\mu_{n}$, wich it's a contradiction. Then $\lambda^{c}(\eta)=0$.

We have that $\eta$ is an equilibrium state for $(f,\phi)$ with $\lambda^{c}(\eta)=0.$ Using similar argument in the proof of the secod item of Theorem A. it is easy to see that $
\eta$ is an ergodic measure. This yields a contradiction to our assumption. This concludes the proof of the dichotomy.
\end{proof}

\section{Proof of Theorem \ref{theo:main.C}}

Theorem \ref{theo:main.C} is a consequence of Theorem \ref{theo:main.A} and \ref{yang-viana}. However, we include a proof which is interesting by itself.
Let $\lambda_{3}<0<\lambda_{2}<\lambda_{1}$ Lyapunov exponents of $A$. Let  $\mu$ be an equilibrium state for $(f,\phi=\psi\circ H)$.
\begin{proposition} \label{Yang}
If $\mu(C)=1$, then $h_{\mu}(f)\leq\lambda_1$.
\end{proposition}  

\begin{lemma}\label{td1} Let $m$ be a probability measure on $\mathbb{R}^{p}\times\mathbb{R}^{q}$, $\pi$  projection onto $\mathbb{R}^{p}$, $m_{t}$ conditional measures of $m$ along the fibers of $\pi$. Define 
$$\gamma(t)=\liminf_{\epsilon\rightarrow 0}\frac{\log m\circ\pi^{-1}B^{p}(t,\epsilon)}{\log\epsilon}$$
and let $\delta\geq0$ be such that at $m-$a.e. $(s,t)$
$$\delta\leq\liminf_{\epsilon\rightarrow 0}\frac{\log m_t B^{q}(s,\epsilon)}{\log\epsilon}$$
Then, at $m-$a.e. $(s,t)$
$$\delta+\gamma(t)\leq\liminf_{\epsilon\rightarrow 0}\frac{\log m B^{p+q}((s,t),\epsilon)}{\log\epsilon}$$
\end{lemma}
\begin{proof} See \cite{Ledrappier-Young}.

\end{proof} By Lemma \ref{ee4}, we have that if $\mu(C)=1$, then $\mu$ is virtually hyperbolic. Let $\nu=H_{*}\mu$ and $R$ be a Markov's rectangle of $A$. We normalize the restriction of $\nu$ on $A$.  Let $\mathcal{ F}^{cu}$ be a typical unstable leaf of $A$.  Consider $R^{cu}=R\cap \mathcal{ F}^{cu}$. Observe that $R^{cu}$ is foliated by strong unstable plaques and also by central (weak unstable) plaques.  Denote by $\nu^{cu}$ the conditional measure of  $\nu$ (normalized and restricted on $R$) on $R^{cu}$. 

Since disintegration of $\nu$ along  central foliation is mono-atomic, we have
$$\nu^{cu}=\int\delta_{a(t)}d\nu^{uu}(t)$$
where $a(t)$ is the unique atom on the central leaf of $t$ and $\nu^{uu}$ is the quotient measure on the quotient of $R^{cu}$ by central plaques. This quotient space can be identified by a strong unstable plaque.
Denote by $\delta^{cu}$ and $\delta^{uu}$ respectively  the dimension of $\nu^{cu}$ and $\nu^{uu}.$ 

\begin{lemma}\label{td3}$\delta^{cu}=\delta^{uu}$.
\end{lemma}
\begin{proof}
The inequality $\delta^{cu}\geq\delta^{uu}$ is immediate by  Lemma $\ref{td1}$. 

Now we prove the other inequality. We denote by $B^{c}(x,\epsilon)$ and $B^{uu}(x,\epsilon)$ the open ball with center $x$ and radius $\epsilon$ respectively on the central leaf  and  on the strong unstable leaf of $x$. Define:

$$D=\lbrace x\in R^{cu}: \exists\alpha>0 \vert \nu^{cu}( B^{uu}(x,\epsilon)\times B^{c}(x,\epsilon))\geq\alpha\nu^{cu}(B^{uu}(x,\epsilon)\times\mathcal{F}^{c}(x) ), \forall\epsilon>0\rbrace$$
Claim: $\nu^{cu}(D)=1$. In fact, we prove that all atom $a(x)$ are in $D$. By definition of conditional measure 
  $$1=\delta_{a(x)}(B^{c}(a(x),\gamma)))=\lim_{\epsilon\rightarrow0}\frac{\nu^{cu}( B^{uu}(a(x),\epsilon))\times B^{c}(a(x),\gamma))}{\nu^{cu}(B^{uu}(a(x),\epsilon)\times\mathcal{F}^{c}(x) )}$$
since  $a(x)$ is the unique atom on the central leaf of $x$, we have that the last equality is hold for all $\gamma>0$. 

Hence,
$$1=\lim_{\epsilon\rightarrow0}\frac{\nu^{cu}( B^{uu}(a(x),\epsilon))\times B^{c}(a(x),\epsilon))}{\nu^{cu}(B^{uu}(a(x),\epsilon)\times\mathcal{F}^{c}(x) )}$$

We take a large enough $n$, then there exist $\epsilon_{0}>0$ such that
$$\frac{n-1}{n}\nu^{cu}(B^{uu}(a(x),\epsilon)\times\mathcal{F}^{c}(x) )<\nu^{cu}( B^{uu}(a(x),\epsilon))\times B^{c}(a(x),\epsilon)),\forall\epsilon<\epsilon_{0}$$ this proves the claim.

If $x\in D$ and since that $h^{c}_{\ast}\nu^{cu}=\nu^{uu}$ ($h^{c}$ is the central holonomy in $R^{cu}$), then
$$\begin{array}{lcl}
 \nu^{cu}(B^{uu}(x,\epsilon)\times B^{c}(x,\epsilon)) &\geq&\alpha\nu^{cu}(B^{uu}(x,\epsilon)\times\mathcal{F}^{c}(x) )\\
 &\geq&\alpha\nu^{uu}(B^{uu}(x,\epsilon))
\end{array}$$
hence,
$$\frac{\log\nu^{cu}(B(x,\epsilon))}{\log\epsilon}\leq\frac{\log\nu^{uu}(B^{uu}(x,\epsilon))+\log\alpha}{\log\epsilon}$$
then, $\delta^{cu}\leq\delta^{uu}$.

\end{proof}
\begin{proof}[Proof of the Proposition \ref{Yang}] By Ledrappier and Young's formula \cite{Ledrappier-Young} and $h_{\nu}(A)=h_{\mu}(f)$, it comes out that $$h_{\mu}(f)=\lambda_{1}\delta_{1}+\lambda_{2}(\delta^{cu}-\delta_{1})$$ where $\delta_{1}$ is the  dimension the measure on the strong unstable leaf. By Lemma \ref{td1}, we have  $\delta_{1}\leq\delta^{cu}$ and  by Lemma \ref{td3}, we have $\delta_{1}\leq\delta^{uu}$. 

So,

$$\delta^{uu}(\lambda_{1}-\lambda_{2})\geq\delta_{1}(\lambda_{1}-\lambda_{2})$$
by Lemma \ref{td3},  $$\delta^{uu}\lambda_{1}\geq\lambda_{1}\delta_{1}+\lambda_{2}(\delta^{cu}-\delta_{1})=h_{\mu}(f)$$
since that $\delta^{uu}\leq1$, we have $h_{\mu}(f)\leq\lambda_{1}$.
\end{proof}
\begin{proof}[Proof of Theorem \ref{theo:main.C}]
We claim that if the potential $\phi$ satisfies the low variational hypothesis of the theorem then the entropy of any equilibrium state of $\phi$ is larger than $\lambda_1.$ To see this it is enough to take $\mu$ as any equilibrium state of $\phi$ and $\eta$ the measure of maximal entropy. 
$$
h_{\mu} + \int \phi d \mu \geq h_{\eta} + \int \phi d \eta = \lambda_1 + \lambda_2 + \int \phi d \eta$$
So,
$$
 h_{\mu} \geq \lambda_1 + \lambda_2 + (\int \phi d\eta - \int \phi d \mu) \geq \lambda_1 + \lambda_2 - (sup \phi - inf \phi) > \lambda_1.
$$
By Proposition \ref{Yang}, we have that all equilibrium state ergodic that satisfies the low variational hypothesis give  zero mass to the union of collapse intervals  $C$. Hence, by item 1 of Theorem \ref{theo:main.A}, if the potential $\phi$ satisfies the low variational hypothesis, then $(f,\phi)$ has a unique equilibrium state. 
\end{proof}

\section{Center Lyapunov exponent and equilibrium state}
\begin{theorem}\label{CenterExponent}
Let $\mu$ be an equilibrium state for $f$ w.r.t. a potencial $\phi=\psi\circ H$. If $\lambda^{c}(\mu)>0$, then $$\lambda_{2}\leq\lambda^{c}(\mu)+\sup\phi-\inf\phi.$$

\end{theorem}
The proof of this result is similar to arguments of Ures (\cite{Ures}, Theorem 5.1) and it is based on  a Pesin-Ruelle-like inequality proved by Y. Hua, R. Saghin and Z.Xia in \cite{Hua-Saghin-Xia}. 
Let $\mathcal{W}$ be a foliation. Let $W_{r}(x)$ be the ball of the leaf $W(x)$  with radius r and centered at $x$. Let $$\chi_{\mathcal{W}}(x,f)=\limsup_{n\rightarrow\infty}\frac{1}{n}\log(vol(f^{n}(W_{r}(x)))$$
$\chi_{\mathcal{W}}(x,f)$ is the volume growth rate of the foliation at $x$. Let $$\chi_{\mathcal{W}}(f)=\sup_{x\in M}\chi_{\mathcal{W}}(x,f)$$
Then, $\chi_{\mathcal{W}}(f)$ is the maximum volume growth rate of $\mathcal{W}$ under $f$.  Let us denote $\chi_{u}(f)=\chi_{\mathcal{W}^u}(f)$ when $f$ is a partially hyperbolic diffeomorphism.

\begin{theorem}[\cite{Hua-Saghin-Xia}]
Let $f$ be a $C^{1+\alpha}$ partially hyperbolic diffeomorphism. Let $\mu$ be an ergodic measure and $\lambda^{c}_{i}(\mu)$ the center Lyapunov exponent of $\mu$. Then,
$$h_{\mu}(f)=\chi_{u}(f)+\sum_{\lambda^{c}_{i}>0}\lambda^{c}_{i}(\mu).$$
\end{theorem}
\begin{proof}[proof of Theorem $\ref{CenterExponent}$] $\chi_{u}(f)\leq \lambda_{1}$, in fact. Since $\mathcal{W}^u$ is 1-dimensional the volume is the length. Then, consider $$\frac{1}{n}\log(vol(f^{n}(W_{r}(x))).$$
Observe that $\chi_{u}(f)=\chi_{u}(\tilde{f})$ where $\tilde{f}$ is any lift of $f$ to universal cover. On the one hand, since $\mathcal{W}^u$ is quasi-isometric, we have that 
$$ \frac{1}{n}\log(vol(\tilde{f}^{n}(W_{r}^u(\tilde{x})))\leq\frac{1}{n}\log(Q\text{diam}(\tilde{f}^{n}(W_{r}^u(\tilde{x}))) $$
for some constant $Q$.
On the other hand, $\tilde{H}(\tilde{f}^{n}(W_{r}^u(\tilde{x})))=\tilde{A}^{n}(\tilde{H}(W_{r}^u(\tilde{x})))$. Let $ C=\text{diam}(\tilde{H}(W_{r}^u(\tilde{x}))) $. Then, $ \text{diam}(\tilde{A}^{n}(\tilde{H}(W_{r}^u(\tilde{x})))) \leq C\exp(n\lambda_{1})$. Since $\tilde{H}$ is at bounded distance from the identity we have that there exists a constant $K$ such that $\text{diam}(\tilde{f}^{n}(W_{r}^u(\tilde{x})))\leq C\exp(n\lambda_{1})+K$. Thus,
$$\frac{1}{n}\log(vol(\tilde{f}^{n}(W_{r}^u(\tilde{x})))\leq\frac{1}{n}(Q(C\exp(n\lambda_{1})+K))$$
Then, $\chi_{u}(f)\leq \lambda_{1}$.

On the other hand, We have that
$$\lambda_{1}+\lambda_{2}+\int\phi d\eta= h_{\eta}(f)+\int\phi d\eta\leq h_{\mu}(f)+\int\phi d\mu$$
where $\eta$ is such that $H_{*}\eta=vol$. Hence,
$$\lambda_{1}+\lambda_{2}+\int\phi d\eta\leq\chi_{u}(f)+\lambda^{c}(\mu)+\int\phi d\mu\leq\lambda_{1}+\lambda^{c}(\mu)+\int\phi d\mu$$
and then,
$$\lambda_{2}\leq \lambda^{c}(\mu)+\int\phi d\mu-\int\phi d\eta\leq\lambda^{c}(\mu)+\sup\phi-\inf\phi.$$
Therefore the theorem is proved.
\end{proof}
By the proof the above theorem, we have the next corollary.
\begin{corollary} If $\mu$ is the unique equilibrium state for $f$ w.r.t a potencial $\phi=\psi\circ H$ with $\sup\phi-\inf\phi<\lambda_{2}$. Then, the center Lyapunov exponent of $\mu$ is positive.
\end{corollary}


\end{document}